\documentclass{article}

\usepackage{amssymb}
\usepackage{amsfonts}
\usepackage{amsmath}

\newtheorem{theorem}{Theorem}

\newtheorem{definition}[theorem]{Definition}

\newtheorem{lemma}[theorem]{Lemma}

\newtheorem{remark}[theorem]{Remark}

\newenvironment{proof}[1][Proof]{\noindent\textbf{#1.} }{\ \rule{0.5em}{0.5em}}
\numberwithin{equation}{section}
\numberwithin{theorem}{section}

\begin{document}
\title{Multitime stochastic maximum principle on curvilinear integral actions}
\author{Constantin Udri\c ste\thanks{University \textit{Politehnica} of \ Bucharest, Faculty of Applied Sciences,
Deptartment of Mathematics-Informatics I; 313, Splaiul Independen\c tei, 060042 Bucharest, Romania; eMail: udriste@mathem.pub.ro},
\,\,Virgil Damian\thanks{University Politehnica of Bucharest, Faculty of Applied Sciences,
Department of Mathematics-Informatics I, Splaiul Independen\c tei 313, 060042 Bucharest, Romania; eMail: vdamian@mathem.pub.ro}.}
\date{}
\maketitle

\begin{abstract}
Based on stochastic curvilinear integrals in the Cairoli-Walsh sense
and in the It\^{o}-Udri\c ste sense, we develop an original theory
regarding the multitime stochastic differential systems. The first
group of the original results refer to the complete integrable
stochastic differential systems, the path independent stochastic
curvilinear integral, the It\^{o}-Udri\c ste stochastic calculus
rules, examples of path independent processes, and volumetric
processes. The second group of original results include the
multitime It\^{o}-Udri\c ste product formula, first stochastic
integrals and adjoint multitime stochastic Pfaff systems. Thirdly,
we formulate and we prove a multitime maximum principle for optimal
control problems based on stochastic curvilinear integral actions
subject to multitime It\^{o}-Udri\c ste process constraints. Our theory requires the
Lagrangian and the Hamiltonian as stochastic $1$-forms.
\end{abstract}



\section{Introduction}

The subject of this paper risen from the intersections of basic
ideas in our works \cite{udriste02} - \cite{udrdam11} and those
well known as stochastic literature \cite{Al:11} - \cite{tud02}, \cite{var} - \cite{zua}. The
principal aim is to solve stochastic optimal control problems based
on curvilinear functionals as actions and stochastic differential
systems as constraints. In this paper we describe how the concepts,
methods and results in \cite{udriste02} - \cite{udrdam11} can be
applied to give a rigorous multitime stochastic model.

There are several reasons why one should learn more about multitime
Wiener process, It\^{o} curvilinear stochastic integrals, multitime
stochastic differential systems, complete integrability conditions, path
independent stochastic curvilinear integral, path independent stochastic processes,
multitime It\^{o}-Udri\c ste product formula and adjoint stochastic
multitime Pfaff systems, optimization problems with curvilinear stochastic
integral functionals, and multitime stochastic maximum principle. They have
a wide range of applications outside mathematics (for example, can be
applied to give a rigorous mathematical models in Finance), there are many
fruitful connections to other mathematical disciplines and these subjects
will give a rapidly developing life of its own as a fascinating research
field with many interesting unanswered questions.

In Section \ref{sect1} we present a brief introduction in
the theory of multitime Wiener processes (e.g., \cite{davar}).
We outline in Section \ref{CaWa} how the introduction of stochastic
curvilinear integrals in the Cairoli-Walsh sense leads to a simple,
intuitive and useful understanding of multitime integration process.
In Section \ref{IntCurb} we define and study the stochastic
curvilinear integrals in our sense. In Section
\ref{SistMultitimeStoch} we study the complete integrable multitime
stochastic differential systems, the path independent stochastic
curvilinear integral, and the multitime stochastic path independent
models. Section \ref{SistAdjuncte} contains original results
regarding the multitime It\^{o}-Udri\c ste product formula and the
adjoint stochastic multitime Pfaff systems and multitime stochastic
first integrals. The stochastic control problems with curvilinear
stochastic integral functionals constrained by multitime stochastic
differential systems are formulated and solved in the Section
\ref{optimizare}. In this context we have obtained a stochastic multitime
maximum principle.

\section{Multitime Wiener process} \label{sect1}

Let $t=\left( t^{\alpha }\right) _{\alpha =\overline{1,m}}\in
\mathbb{R} _{+}^{m}$ be a multi-parameter of evolution or
\textit{multitime} and let $ \Omega _{0T}\subset \mathbb{R}_{+}^{m}$
be the\textit{\ parallelepiped} fixed by the diagonal opposite
points $0=\left( 0,...,0\right) $ and $ T=\left(
T^{1},...,T^{m}\right) $, equivalent to the closed interval $0\leq
t\leq T$, via the \textit{product order} on $\mathbb{R}_{+}^{m}$,
defined by
\begin{equation*}
\left( t_{1}^{1},t_{1}^{2},...,t_{1}^{m}\right) \leq (<)~\left(
t_{2}^{1},t_{2}^{2},...,t_{2}^{m}\right) \text{ if }t_{1}^{\alpha }\leq
(<)~t_{2}^{\alpha }\text{, for all }\alpha =\overline{1,m}\text{.}
\end{equation*}

Let $\left( \Omega ,\mathcal{F},\mathbb{P}\right)$ be a probability space
endowed with a \textit{complete}, \textit{increasing} and 
\textit{right-continuous} filtration (a complete natural history) 
$\left\{ \left( \mathcal{F}_{t}\right) _{t}:t\in \Omega _{0T}\right\}$.
Such a probability space 
$\left( \Omega ,\mathcal{F},\left( \mathcal{F}_{t}\right) _{t\in \Omega _{0T}},\mathbb{P}\right) $ 
is called \textit{filtred probability space}. Let $I$ be any subset of 
$\left\{1,...,m\right\} $ and let us denote by $\mathcal{F}_{t}\left( I\right) $ the
$\sigma -$algebra generated by the $\sigma -$algebras 
$\mathcal{F}_{\widetilde{t}}$, where $\widetilde{t}^{\alpha }\leq t^{\alpha }$, 
for all $\alpha \in I$. We say that the filtration satisfies the \textit{conditional
independence} property if for all bounded random variables $X$, all $t\in
\mathbb{R}^{m}$ and $I\subset J\subset \left\{ 1,...,m\right\} $,
\begin{equation}
\mathbb{E}\left[ X\mid \mathcal{F}_{t}\left( J\right) \right] =\mathbb{E}
\left[ \mathbb{E}\left[ X\mid \mathcal{F}_{t}\left( J\right) \right] \mid
\mathcal{F}_{t}\left( J\setminus I\right) \right] .  \label{comutare}
\end{equation}
This property implies that the conditional expectations with respect to 
$\mathcal{F}_{t}\left( I\right) $ and $\mathcal{F}_{t}\left( J\setminus
I\right) $ \textit{commute}.

For a multitime process $x=\left( x\left( t,\omega \right) \right) _{t\in \Omega_{0T}}
$, the \textit{increment} of $x$ on an interval $(t_{1},t_{2}]\subset
\Omega_{0T}$, is given by (\cite[pp. 10]{dozzi})
\begin{equation}
x\left( (t_{1},t_{2}]\right) =\sum\limits_{i\left( 1\right) =1}^{2}\cdots
\sum\limits_{i\left( m\right) =1}^{2}\left( -1\right) ^{\sum\limits_{\alpha
=1}^{m}i\left( \alpha \right) }x\left( t_{i\left( 1\right)
}^{1},...,t_{i\left( m\right) }^{m}\right).  \label{cresterea_procesului}
\end{equation}
For simplicity, here and in the whole paper, we will denote by $\xi_{t}$ the
random variable $\xi \left( t,\omega \right) $ at each multitime $t $.

\begin{definition}
\emph{(Martingale)} Let $x=\left( x_{t}\right) _{t\in \Omega _{0T}}$ be a multitime $%
\mathcal{F}_{t}-$adapted process.

\textit{(i)} The process $x$ is called \textit{weak martingale} if $\mathbb{E}\left[
x\left( (t,s]\right) \mid \mathcal{F}_{t}\right] =0$, for all $t,s\in
\mathbb{R}_{+}^{m}$, such that $t\,\leq s$.

\textit{(ii)} The process $x$ is called \textit{martingale} if $\mathbb{E}\left[
x_{s}\mid \mathcal{F}_{t}\right] =x_{t},$ for all $t,s\in \mathbb{R}_{+}^{m}$,
such that $t\,\leq s$.
\label{def_martingal}
\end{definition}

Clearly, every martingale is a weak martingale \cite{dozzi}.

\begin{definition}
A multitime, real-valued, and right-continuous process 
$A=\left( A_{t}:t\in \mathbb{R}_{+}^{m}\right) $ 
is said to be \textit{increasing} if $A_{t}=0$ $\mathbb{P-}
$a.s. for $t\in \mathbb{R}_{+}^{m}$ and if $A\left( (t,s]\right) \geq 0$,
for all subintervals $(t,s]\subset \mathbb{R}_{+}^{m}$ 
(see (\ref{cresterea_procesului})). \label{incr_proc}
\end{definition}

\begin{theorem}[{\protect\cite[Prop. 8, pp. 41]{dozzi}}]
\label{teorema}If $x$ is a squared-integrable martingale, then there exists
an increasing process $\left( A_{t}\right) $ such that $\left(
x_{t}^{2}-A_{t}\right) $ is a weak martingale.
\end{theorem}

\begin{theorem}[{\protect\cite[Prop. 8, pp. 41]{dozzi}}]
Let $x$ be a strong continuous martingale such that $\mathbb{E} \left[ |x_t|^4
\right] < \infty$. Then there exists an increasing
$\mathcal{F}_{t}-$ previsible process $\left[ x\right] _{t}$ such that $\left( x_{t}^{2}-\left[
x\right] _{t}\right) $ is a martingale.
\end{theorem}

Let $\mathcal{B}$ be the Borel $\sigma -$field of $\mathbb{R}_{+}^{m}$ and
let $\nu $ denote the \textit{Lebesgue measure}. 
For $t=\left(t^{1},...,t^{m}\right) $, set $\Omega _{0t}=\left[ 0,t^{1}\right]
\times \left[ 0,t^{2}\right] \times ...\times \left[ 0,t^{m}\right] $ and
\begin{equation*}
W\left( t\right) =W\left( t^{1},...,t^{m}\right) =W\left( \Omega_{0t}\right).
\end{equation*}
This defines a multitime mean-zero Gaussian process \cite[Ch. 5]{davar}
$W=W\left( t\right) _{t\in \Omega _{0T}}$
with covariance
\begin{equation*}
\mathbb{E}\left[ W_{t_{1}}W_{t_{2}}\right] =\nu \left( \left[ 0,t_{1}\right]
\cap \left[ 0,t_{2}\right] \right) =\prod\limits_{\alpha =1}^{m}\min
\left\{t_{1}^{\alpha },t_{2}^{\alpha }\right\}.
\end{equation*}

\begin{definition}
A stochastic process of the form $\left( W_{t}:t\in \mathbb{R}_{+}^{m}\right) $ is
called \textit{multitime Wiener process }(starting at zero) or
\textit{Brownian sheet} if $W_{0}=0$ and if $W_{t}$ is a gaussian process
with $\mathbb{E}\left[ W_{t}\right] =0$ and for $t_{1}=\left( t_{1}^{\alpha
}\right) _{\alpha =\overline{1,m}}$, $t_{2}=\left( t_{2}^{\alpha }\right)
_{\alpha =\overline{1,m}}$,
\begin{equation*}
\mathbb{E}\left[ W_{t_{1}}W_{t_{2}}\right] =\prod\limits_{\alpha =1}^{m}\min
\left\{ t_{1}^{\alpha },t_{2}^{\alpha }\right\}.
\end{equation*}
\end{definition}

\begin{definition}
The multitime stochastic process $\left( W_{t}:t\in \Omega _{0T}\right) $ is called a
\textit{\ }$\mathcal{F}_{t}$-\textit{Wiener process} if, in addition,
$\mathbb{E}\left[ W_{s}\mid \mathcal{F}_{t}\right] =W_{t}\text{,}$
for all $t,s\in \mathbb{R}_{+}^{m}$, such that $t\,\leq s$.
\end{definition}

A first example of martingale is the multitime Wiener process.

\textbf{Hypothesis right-left (RL)} Suppose a sample sheet 
$x:\Omega_{0T}\longrightarrow \mathbb{R} $ is continuous from the right and
bounded from the left at every point. That is, for every $t_{0}\in \Omega
_{0T}$, $t\downarrow t_{0}$, implies $x\left( t\right) \longrightarrow
x\left( t_{0}\right) $ and for $t\uparrow t_{0}$, $\lim_{t\uparrow
t_{0}}x\left( t\right) $ exists, but need not be $x\left( t_{0}\right) $. We
use only stochastic processes $x$ where almost all sample sheets have the
\textbf{RL} property.

\section{Stochastic curvilinear integrals in \\the Cairoli-Walsh sense} \label{CaWa}

Let $\Gamma$ be an oriented piecewise $C^{1}$ curve in $\mathbb{R}_{+}^{m}$ given by parametric
representation $t=\gamma (\tau )$, $\tau \in \lbrack 0,1]$. The curve $\widehat{\Gamma}$ of the parametric
representation $t=\widehat{\gamma }(\tau )=\gamma (1-\tau )$,
$\tau \in \lbrack 0,1]$ has opposite orientation.

\begin{definition}
A piecewise $C^{1}$ curve is called of \textit{pure type} if each component
of the tangent vector field $\displaystyle\frac{d\gamma}{d\tau}$ preserves its sign.
\end{definition}

Given a process $W=W\left( t\right) _{t\in \Omega _{0T}}$
and a pure type curve $\Gamma$, we can define $m$ processes
$W_{\alpha }^{\Gamma }$, $\alpha =\overline{1,m}$, on $\Gamma$,
which may be thought of as coming from increments in the direction of each
$Ox_{\alpha }$ axis. A suggestive notation for this would be
$dW_{\alpha }^{\Gamma }=\partial _{\alpha }W$.

If $\displaystyle\frac{d\gamma^{\alpha}}{d\tau}(\tau )\geq 0$, i.e., the component
$\gamma^{\alpha }$ is nondecreasing function and if $\gamma (0)=t_{0}$
and $t=\gamma (\sigma)\in \Gamma $, then we introduce the subset
\[
D_{t}^{\alpha }=\{s\in \mathbb{R}_{+}^{m}\mid t_{0}^{\alpha }<s ^{\alpha }\leq t^{\alpha },
~0\leq s^{\beta }\leq \gamma ^{\beta }(\tau ), \beta \not= \alpha,\, 0<\tau \leq \sigma\}.
\]
For a squared-integrable martingale $W$ and a curve $\Gamma$ of pure
type joining the initial point $t_{0}$ and the final point $t_{f}$, with
$\displaystyle\frac{d\gamma^{\alpha}}{d\tau}(\tau )\geq 0$, we define
$W_{\alpha }^{\Gamma}(t_{f})=W(D_{t_{f}}^{\alpha })$ and
$W_{\alpha }^{\Gamma }(t)=\mathbb{E\{}W(D_{t_{f}}^{\alpha })\mid \mathcal{F}_{t}^{\alpha }\}$, $t\in \Gamma $.
Then each $W_{\alpha }^{\Gamma }=\{W_{\alpha }^{\Gamma }(t),\mathcal{F}_{t}^{\alpha },t\in \Gamma \}$
is a one-parameter square-integrable martingale. Consequently, one can define the It\^{o} curvilinear integral of a process $\phi
=\left\{ \phi \left( t\right) :t\in \Omega _{0T}\right\}$ with respect to $W_{\alpha}^{\Gamma }$, in the usual way, and denote it by
\[
\int_{\Gamma }\phi(t)\,\, \partial _{\alpha }W.
\]
If the component $\alpha$ of the tangent vector field satisfy the condition $\displaystyle\frac{d\gamma^{\alpha}}{d\tau}(\tau )\leq 0$
(nonincreasing function), then we define
\[
\int_{\Gamma }\phi(t)\,\, \partial _{\alpha }W = -\int_{\widehat{\Gamma }}\phi(t)\,\,\partial _{\alpha }W.
\]

Finally, if $\Gamma$ is of pure type curve, we let
\[
\int_{\Gamma }\phi(t) \,\,\partial W=\sum_{\alpha =1}^{m}\int_{\Gamma }\phi(t)\,\,\partial _{\alpha }W.
\]

\section{Stochastic curvilinear integrals in \\the It\^{o}-Udri\c ste sense} \label{IntCurb}

In our theory we need a curve
\begin{equation}
\gamma :\left[ 0,1\right] \longrightarrow \Omega _{0T}\subset \mathbb{R}
_{+}^{m}, \,t=t\left( \tau \right), \,\tau\in\left[ 0,1\right],
\end{equation}
where $\tau$ is the curvilinear abscissa. The curve $\gamma$ is called
{\it increasing} if
\begin{equation*}
\gamma(\tau)\leq \gamma(\tau^{\prime})\,\, \hbox{in}\,\, \mathbb{R}%
_{+}^{m},\,\, \hbox{if}\,\, \tau \leq \tau^{\prime}\,\,\hbox{in} \,\,[0,1].
\end{equation*}
The curve $\gamma$ is called {\it piecewise $C^1$} if there exists a curve $%
\gamma^{\prime} :\left[ 0,1\right] \longrightarrow \Omega _{0T}$ with
finitely many discontinuities, satisfying the (RL) hypotheses, so that
\begin{equation*}
\gamma(\tau) = \int_0^\tau \gamma^{\prime}(\lambda) d\lambda,\,\,\forall
\tau \in [0,1].
\end{equation*}

Let $t=\left( t^{1},...,t^{m}\right) \in \Omega _{0T}$. Let $\left(
W_{t}\right) _{t\in \Omega _{0T}}$ be a multitime Wiener process. Then (%
\cite[Th. 3, pp. 45]{dozzi}) $W^{2}$ decomposes into the sum of a weak
martingale and an increasing process (see Definition \ref{def_martingal} and
Definition \ref{incr_proc}). According to Dozzi (\cite[Lemma 6, pp. 151]%
{dozzi}), the increasing process associated with $\left( W_{t}\right) _{t\in
\Omega _{0T}}$ is denoted by $\left\langle W\right\rangle
_{t}=t^{1}t^{2}...t^{m}$ and it means that for $h=\left(
h^{1},...,h^{m}\right) $, we have
\begin{equation*}
\mathbb{E}\left[ W_{t+h}^{2}-\prod\limits_{\alpha =1}^{m}\left( t^{\alpha
}+h^{\alpha }\right) \mid \mathcal{F}_{t}\right] =W_{t}^{2}-\prod\limits_{%
\alpha =1}^{m}t^{\alpha }.
\end{equation*}
In other words, for the multitime $t=\left( t^{1},...,t^{m}\right) \in \Omega _{0T}$, the stochastic process
\begin{equation*}
\left( W_{t}^{2}-\prod\limits_{\alpha =1}^{m}t^{\alpha}\right) _{t\in \Omega
_{0T}}
\end{equation*}
is a (continuous) martingale. This leads to the construction of stochastic
curvilinear integral with respect to $W$. The case $m=2$ was explicitly
given in \cite[\S 7, pp. 157]{cairoli&walsh}.

\begin{definition}
\label{definitia.integralei} Let $\phi =\left\{ \phi \left( t\right) :t\in
\Omega _{0T}\right\} $ be a real $\mathcal{F}_{t}-$predictible process, such
that
\begin{equation}
\mathbb{E}\left[ \int_{\gamma _{0T}}\phi _{t}^{2}\,\,d\left\langle
W\right\rangle _{t}\right] <\infty .
\end{equation}
Let ${\gamma _{0T}}$ be a $C^1$ increasing curve. The real number
\begin{equation}
I\left( \phi \right) =\int_{\gamma _{0T}}\phi \left( t\right)
dW_{t}=\int_{0}^{1}\phi \left( \gamma \left( \tau \right) \right) dB_{\tau}
\label{intcurv000}
\end{equation}
is called the \textit{stochastic curvilinear integral} of the process $\phi
=\left\{ \phi \left( t\right) :t\in \Omega _{0T}\right\}$ along the curve $%
\gamma $ with respect to $(W_{t})_{t\in \Omega _{0T}}$, where $B_{\tau}
\overset{def}{=}W_{\gamma \left( \tau \right) }$ is stochastic processes normal distributited of mean $0$ and variance $vol(\Omega_{0\gamma(\tau)})$.
\end{definition}

For $m=2$, see \cite{cairoli&walsh}.

\begin{definition}
Let $\phi =\left\{ \left( \phi _{a}\left( t\right) \right) _{a=\overline{1,d}%
}:t\in \Omega _{0T}\right\} $ be an $\mathcal{F}_{t}-$predictible process
with values in $\mathbb{R}^{d}$. Let ${\gamma _{0T}}$ be a $C^1$ increasing
curve. The real number
\begin{equation}
I\left( \phi \right) =\int_{\gamma _{0T}}\phi _{a}\left( t\right)
dW_{t}^{a}=\int_{0}^{1}\phi _{a}\left( \gamma \left( \tau \right) \right)
dB_{\tau }^{a}  \label{intcurv001}
\end{equation}
(i.e., a sum of It\^{o} classical integrals) is called the \textit{%
stochastic curvilinear integral} of process $\phi =\left\{ \left(
\phi_{a}\left( t\right) \right) _{a=\overline{1,d}} \right\}$, 
along the curve $\gamma $ with respect to the Wiener
process $W=(W_{t}^{a})_{a=\overline{1,d}}$, where $B_{\tau }^{a}\overset{def}%
{=}W_{\gamma \left( \tau \right) }^{a}$, for each $a=\overline{1,d}$, is a stochastic processes as in the above definition.
\end{definition}

Here and in the whole article, for any given Euclidean space $H$, we denote
by $\left\langle \cdot ,\cdot \right\rangle $ (resp. $\left\Vert \cdot
\right\Vert $) the \textit{inner product} (resp. \textit{norm}) of $H$.
Also, we use Einstein summation convention.

\begin{remark}
\label{remarca_01} With the previous definition (\ref{intcurv001}), we have,
by It\^{o}'s isometry property
\begin{equation*}
\mathbb{E}\left[ \left\vert I\left( \phi \right) \right\vert ^{2}\right] =
\mathbb{E}\left[ \int_{\gamma _{0T}}\Vert \phi (t)\Vert ^{2}~d(t^{1}...t^{m})%
\right]
\end{equation*}
\begin{equation*}
=\mathbb{E}\left[ \int\limits_{0}^{1}\Vert \phi \left( t\left( \tau\right)
\right) \Vert ^{2}\frac{d}{d{\tau }}\left( t^{1}\left( \tau
\right)...t^{m}\left( \tau \right) \right)d\tau \right].
\end{equation*}
\end{remark}

A particular case ($n=1, m=2$) of the previous formula is proved in \cite%
{cairoli&walsh}. It is obvious that $\mathbb{E}\left[ I\left( \phi \right) %
\right] =0$.

\section{Multitime stochastic differential systems} \label{SistMultitimeStoch}

Let us change the single-time approach of stochastic theory (see, for example \cite{Ok:03}) 
to a new approach issuing from the papers of the first author. We use a multitime parameter of evolution 
$t=\left( t^{1},...,t^{m}\right) \in \Omega _{0T}$ and we introduce the multitime
stochastic differential systems ($mSDS$). For that, let $f(t,x_{t})$
$ =\left( f_{\alpha }^{i}\left( t,x_{t}\right) \right)
_{\substack{ i=\overline{1,n}  \\ \alpha =\overline{1,m}}}$ be an $n\times m$
matrix of previsible processes with {\cite[pp. 73]{dozzi}}
\begin{equation*}
\mathbb{E}\left[ \int_{\gamma _{0t}}\left\Vert f\left( s(\tau),x_{s(\tau)}\right)
\right\Vert d\tau\right] <\infty \text{ for all }t\in \Omega _{0T},
\end{equation*}%
where $\tau$ is the curvilinear abscissa on $\gamma _{0t}$, let $g\left(
t,x_{t}\right) =\left( g_{a}^{i}\left( t,x_{t}\right) \right) _{\substack{ i=%
\overline{1,n}  \\ a=\overline{1,d}}}$ be an $n\times d$ matrix of
previsible processes, such that
\begin{equation*}
\mathbb{E}\left[ \int_{\gamma _{0t}}\Vert g\left( s,x_{s}\right) \Vert
^{2}d(s^{1}...s^{m})\right] <\infty \text{ for all }t\in \Omega _{0T},
\end{equation*}%
and let $W=\left( W_{t}^{a}\right) _{t}$, $a=\overline{1,d}$, be a multitime
Wiener process with values in $\mathbb{R}^{d}$.

\begin{definition}
\label{def00001} Let $i=\overline{1,n}$, $\alpha =\overline{1,m}$ and $a=\overline{1,d}$.
The multitime stochastic differential system
\begin{equation}
dx^{i}_t=f_{\alpha }^{i}\left( t,x_t\right) dt^{\alpha
}+g_{a}^{i}\left(t,x_t\right) dW^{a}_t,\text{ }t\in \Omega _{0T}
\label{0201}
\end{equation}
is called \textit{It\^ o - Pfaff stochastic system}.
\end{definition}

The coefficients $f_{\alpha }^{i}\left( t,x_t\right) $ are called \textit{drift coefficients}
and $g_{a}^{i}\left( t,x_t\right) $ are the \textit{diffusion coefficients}.

\begin{definition}
A multitime stochastic differential system is called \textit{completely
integrable} if there exists an $m$-sheet $x(t)$ satisfying the stochastic
integral relation
\begin{equation*}
x^{i}\left(t\right) =x^{i}\left( 0\right) +\int_{\gamma
_{0t}}f_{\alpha}^{i}\left(s,x_{s}\right) ds^{\alpha }+ \int_{\gamma
_{0t}}g_{a}^{i}\left(s,x_{s}\right) dW_{s}^{a},\,\,\,t\in \Omega_{0T},
\end{equation*}
independent of the selection of the $C^1$ increasing curve ${\gamma _{0t}}$.
\end{definition}

Thus, the right hand side of an It\^ o - Pfaff stochastic system is well-defined as a
stochastic curvilinear integral, under suitable assumptions on the functions $f_{\alpha}^{i}$ and $g_{a}^{i}$.

The first integral can be interpreted as an ordinary curvilinear integral.
The second integral, i.e., the stochastic curvilinear integral, cannot be
treated as such, since the sheet-wise $W_{t}$ is nowhere differentiable. Of
course, the complete integrability conditions for the first curvilinear
integral are contained in the classical books while, for the stochastic
curvilinear integral, only partial results can be found in \cite{cairoli&walsh}, \cite{dozzi}.

If a multitime It\^ o - Pfaff stochastic differential system is not completely integrable,
given the increasing $C^1$ curve ${\gamma _{0t}}: s^\alpha = s^\alpha
(\tau),\,\,\tau \in [0,\tau_0],\,\, s(0) = 0$, a curve $x(\tau)$ which
satisfies
\begin{equation*}
x^{i}(\tau) =x^{i}(0) +\int_0^\tau f_{\alpha}^{i}\left(s(\lambda),x(\lambda)
\right) \frac{ds^\alpha}{d\lambda}d\lambda + \int_0^\tau
g_{a}^{i}\left(s(\lambda),x(\lambda) \right) dB_\lambda^{a}
\end{equation*}
is called \textit{solution}.

Having in mind some ideas in \cite{cairoli&walsh}, \cite{dozzi}, completing with our ideas,
let us praise a fundamental path independent stochastic curvilinear integral.

\subsection{Path independent stochastic curvilinear integral}

Here is for the first time when is presented a curvilinear stochastic
integral independent of the path.

\begin{theorem}
Let ${\gamma_{0t}}$ be an increasing curve in $%
\Omega_{0T}$. The stochastic curvilinear integral (primitive) $%
\int_{\gamma_{0t}}W_s dW_s$ has the value $\frac{W_t^2 - W_0^2}{2} - \frac{1%
}{2}\,\,t^1\cdots t^m$. It can be written as
\begin{equation*}
\frac{1}{2} W_t^2 = \frac{1}{2} W_0^2 + \frac{1}{2}\int_{\gamma_{0t}}d(s^1%
\cdots s^m) + \int_{\gamma_{0t}}W_s dW_s.
\end{equation*}
Obviously, $W_0 = 0$ and the stochastic curvilinear integral $%
\int_{\gamma_{0t}}W_s dW_s$ is path independent.
\end{theorem}

\begin{remark} Our point of view requires the volume written as a special
curvilinear integral,
\begin{equation*}
\int_{\gamma_{0t}}d(s^1\cdots s^m) =
vol(\Omega_{0t}),\,\,\int_{\gamma_{st}}d(\tau^1\cdots \tau^m) =
vol(\Omega_{0t}) - vol(\Omega_{0s}),\,\,s\leq t.
\end{equation*}
\end{remark}

\textbf{Proof (Ionel \c Tevy)} For simplicity, we refer to $m=2$. Let $%
\mathcal{M}$ be the random measure in $\mathbb{R}_{+}^{2}$ which assigns to
each Borel set $A$ a Gaussian random variable of mean zero and variance $\mu
( A) $, where $\mu $ is Lebesgue measure and which assigns independent
random variables to disjoint sets.

Define a process $(W(z) :z\in \mathbb{R}_{+}^{2}) $ by $W(z) =W( \Omega
_{0z}) $, where $\Omega _{0z}$ is the rectangle whose lower left hand corner
is the origin $O(0,0)$ and whose upper right hand corner is $z=(s,t)$. The
process $(W(z) :z\in \mathbb{R}_{+}^{2})$ is called a \textit{two-parameter Wiener process}.

For $z_{1}( s_{1},t_{1})$ and $z_{2}(s_{2},t_{2})$ with $z_{1}<z_{2}$, note $%
L_{z_{1}z_{2}}=\Omega _{0z_{2}}\setminus \Omega_{0z_{1}}$. Let $[a,b] \times
[c,d] $ be a rectangle in $\mathbb{R}_{+}^{2}$, and let
\begin{equation*}
P^{n}=\{ a=s_{0}^{n}<s_{1}^{n}<...<s_{m_{n}}^{n}=b\},\,\, Q^{n}=\{
c=t_{0}^{n}<t_{1}^{n}<...<t_{m_{n}}^{n}=d\}
\end{equation*}
be partitions of the segment $[a,b]$ and, respectively, $[c,d]$ with $\vert
P^{n}\vert \longrightarrow 0$, $\vert Q^{n}\vert \longrightarrow 0$ as $%
n\longrightarrow \infty $. Let us denote $\alpha =(a,c)$, $\beta =(b,d)$, $%
z_{k}^{n}=\left( s_{k}^{n},t_{k}^{n}\right) $, $k=\overline{1,m_{n}}$. Then
\begin{equation*}
\vert L^{n}\vert =\underset{k}{\max }\left( area\left(
L_{z_{k}^{n}z_{k+1}^{n}}\right) \right) \longrightarrow 0\text{ as }%
n\longrightarrow \infty \text{.}
\end{equation*}

\begin{lemma}[Quadratic variation]
The limit
\begin{equation*}
\sum_{k=0}^{m_{n}-1}\left( W\left( z_{k+1}^{n}\right) -W\left(
z_{k}^{n}\right) \right) ^{2}\longrightarrow area\left( L_{\alpha \beta
}\right) ,
\end{equation*}
as $n\longrightarrow \infty$, holds in $L^{2}$.
\end{lemma}

\begin{proof}
Set
\begin{equation*}
R_{n}=\sum_{k=0}^{m_{n}-1}\left( W\left( z_{k+1}^{n}\right) -W\left(
z_{k}^{n}\right) \right) ^{2}.
\end{equation*}%
Then%
\begin{equation*}
R_{n}-area\left( L_{\alpha \beta }\right) =\sum_{k=0}^{m_{n}-1}\left[ \left(
W\left( z_{k+1}^{n}\right) -W\left( z_{k}^{n}\right) \right) ^{2}-area\left(
L_{z_{k}^{n}z_{k+1}^{n}}\right) \right] .
\end{equation*}%
Denoting,
\begin{equation*}
\rho_k = \left( W\left(z_{k+1}^{n}\right) -W\left( z_{k}^{n}\right) \right)
^{2}-area\left(L_{z_{k}^{n}z_{k+1}^{n}}\right)
\end{equation*}
we find
\begin{equation*}
\mathbb{E}\left[ \left( R_{n}-area\left( L_{\alpha \beta }\right) \right)
^{2}\right] =\sum_{k=0,j=0}^{m_{n}-1}\mathbb{E}(\rho_k \rho_j).
\end{equation*}%
For $k\neq j$, the term in the double sum equals $0$, according to the
independent increments, as $W(v)-W(u)$ and $\mathcal{N}\left( 0,area\left(
L_{uv}\right) \right) $. Hence%
\begin{equation*}
\mathbb{E}\left[ \left( R_{n}-area\left( L_{\alpha \beta }\right) \right)
^{2}\right] =\sum_{k=0}^{m_{n}}\mathbb{E}\left[ \left( Y_{k}^{2}-1\right)
^{2}\left[ area\left( L_{uv}\right) \right] ^{2}\right] ,
\end{equation*}%
where%
\begin{equation*}
Y_{k}=Y_{k}^{n}=\frac{W\left( z_{k+1}^{n}\right) -W\left( z_{k}^{n}\right) }{%
\sqrt{Area\left( L_{z_{k}^{n}z_{k+1}^{n}}\right)}},
\end{equation*}
is a Gaussian process $\mathcal{N}\left( 0,1\right)$.
Therefore, for some constant $C$, we have%
\begin{equation*}
\mathbb{E}\left[ \left( R_{n}-area\left( L_{\alpha \beta }\right) \right)
^{2}\right] \leq C\sum_{k=0}^{m_{n}-1}\left[ area\left(
L_{z_{k}^{n}z_{k+1}^{n}}\right) \right] ^{2}\leq
\end{equation*}%
\begin{equation*}
\leq C\left\vert L^{n}\right\vert area\left( L_{\alpha \beta }\right)
\longrightarrow 0\text{, as }n\longrightarrow \infty .
\end{equation*}
\end{proof}

Let $\Gamma^n = \{0<z_1<...<z_{m_n} = z\}$ be a partition of the curve $%
\gamma_{0z}$, i.e., $z^n_k(s^n_k,t^n_k)\in \gamma_{0z}$, and $|\Gamma^n| =
\max_k \hbox{aria}\,\, (\Omega_{z^n_{k+1}}\setminus \Omega_{z^n_{k}})$.
Note, according It\^ o definition,
\begin{equation*}
R_n = \sum_{k=0}^{m_n-1}\,W(z^n_k)\left(W(z^n_{k+1}) - W(z^n_k)\right).
\end{equation*}
Then, we have
\begin{equation*}
R_n =\frac{W^2(z)}{2} - \frac{1}{2} \sum_{k=0}^{m_n-1}\,\left(W(z^n_{k+1}) -
W(z^n_k)\right)^2.
\end{equation*}
When $|\Gamma^n|\to 0,$ as $n \to \infty$, based on the foregoing Lemma, we
find
\begin{equation*}
\sum_{k=0}^{m_n-1}\,\left(W(z^n_{k+1}) - W(z^n_k)\right)^2 \to st
\end{equation*}
in $L^2$, as $n \to \infty$, and the result is established.

The following examples are based on the complete integrability notion for
the stochastic curvilinear integrals and on \textit{It\^{o}-Udri\c ste
stochastic calculus rules} \cite{udriste02}
\begin{equation*}
dW_{t}^{a}~dW_{t}^{b}= \delta ^{ab}\,
c_\alpha(t) dt^\alpha, dW_{t}^{a}~dt^\alpha=dt^\alpha~dW_{t}^{a}=0,
dt^\alpha\,dt^\beta=0,
\end{equation*}
for any $a,b=\overline{1,d}; \alpha, \beta = \overline{1,m}$, where $\delta
^{ab}$ is the \textit{Kronecker symbol}, $c_\alpha(t)= \frac{\partial}{%
\partial t^\alpha}(t^1\cdots t^m)$ and the tensorial product $\delta
^{ab}\,c_\alpha(t)$ represents the \textit{correlation coefficients}.

\subsection{Examples of path independent processes}

1) \textbf{Stock prices}. The idea to reconsider applications in Finance via the
multitime stochastic calculus was inspired by the work given in \cite{Ok:03}.
Let $(P_{t})_{t\in \mathbb{R}_{+}^{2}}$ denote the
price of a stock at two-time $t=(t^{1},t^{2})\in \mathbb{R}_{+}^{2}$, where $t^1$ means time
and $t^2$ represents a "space" variable (as example, showing the price evolution as function of the distance 
between supplier and seller). In this way, the stochastic perturbations involved in price dynamics
are modelled by a \textit{two-time (time-space) Brownian sheet}. We
model the evolution of the price $P_{t}$ supposing that the relative change 
$\frac{dP_{t}}{P_{t}}$ in price $P_{t}$ is involved in the SDE
\begin{equation*}
\frac{dP_{t}}{P_{t}}=\mu _{\alpha }dt^{\alpha }+\sigma _{\alpha }dW^{\alpha},
\end{equation*}
for constant drift vector $\mu _{\alpha }>0$ and constant diffusion vector $
\sigma _{\alpha }$. Hence
\begin{equation*}
{dP}_{t}=P_{t}\mu _{\alpha }dt^{\alpha }+P_{t}\sigma _{\alpha }dW^{\alpha }.
\end{equation*}%
Using It\^{o}-Udri\c{s}te formula
\begin{equation*}
d[\ln \,P_{t}]=\frac{dP_{t}}{P_{t}}-\frac{1}{2}\frac{P_{t}^{2}\sigma
_{\alpha }\sigma _{\beta }\delta ^{\alpha \beta }c_{\lambda }dt^{\lambda }}{
P_{t}^{2}}=\left( \mu _{\lambda }-\frac{1}{2}\sigma _{\alpha }\sigma _{\beta
}\delta ^{\alpha \beta }c_{\lambda }(t)\right) dt^{\lambda }+\sigma _{\alpha
}dW^{\alpha },
\end{equation*}%
we find
\begin{equation*}
P_t=p_{0}e^{\mu _{\lambda }t^{\lambda }-\frac{1}{2}\sigma _{\alpha }\sigma
_{\beta }\delta ^{\alpha \beta }t^{1}t^{2}+\sigma _{\alpha }W^{\alpha }}.
\end{equation*}%
The price $P_t$ is always positive, if $p_{0}>0.$ Since
\begin{equation*}
P_t=p_{0}+\int_{\gamma _{0t}}P_s\mu _{\alpha }ds^{\alpha }+\int_{\gamma
_{0t}}P_s\sigma _{\alpha }dW^{\alpha }
\end{equation*}%
and $\mathbb{E}\left[ \int_{\gamma _{0t}}P_s\sigma _{\alpha }dW^{\alpha }%
\right] =0$, we find
\begin{equation*}
\mathbb{E}[P_{t}]=p_{0}+\int_{\gamma _{0t}}\mathbb{E}[P_{s}]\mu _{\alpha}ds^{\alpha }.
\end{equation*}%
Hence
\begin{equation*}
\mathbb{E}[P_t]=p_{0}e^{\mu _{\alpha }t^{\alpha }},\,t^{\alpha }\geq 0.
\end{equation*}%
This expected value of the stock price is the same with the deterministic
solution of the completely integrable Pfaff equation
\begin{equation*}
{dP}_{t}=P_{t}\mu _{\alpha }dt^{\alpha }.
\end{equation*}

2) Let $t=(t^{1},t^{2})\in \mathbb{R}_{+}^{2}$, $x\in \mathbb{R}$ and
$g_{\alpha }:\mathbb{R}_{+}^{2}\rightarrow \mathbb{R},\,\alpha =1,2$ be two
continuous functions. Let $c_{1}(t)=t^{2},c_{2}(t)=t^{1}$ and ${\gamma _{0t}}$
be an increasing curve in $\Omega _{0T}$. The unique solution of the stochastic differential
equation $dx_{t}=x_{t}g_{\alpha }(t)dW_{t}^{\alpha },\,x(0)=1$ is
\begin{equation*}
x_t=e^{-\frac{1}{2}\int_{\gamma _{0t}}g_{\alpha }(s)g_{\beta }(s)\delta
^{\alpha \beta }c_{\lambda }(s)ds^{\lambda }+\int_{\gamma _{0t}}g_{\alpha
}(s)dW_{s}^{\alpha }},
\end{equation*}%
if the (usual and stochastic) curvilinear integrals are path independent,
i.e., $g_{\alpha }(t)=h_{\alpha }(t^{1}t^{2})$. To verify the solution, we
note that the stochastic process
\begin{equation*}
y_t=-\frac{1}{2}\int_{\gamma _{0t}}g_{\alpha }(s)g_{\beta }(s)\delta
^{\alpha \beta }c_{\lambda }(s)ds^{\lambda }+\int_{\gamma _{0t}}g_{\alpha
}(s)dW_{s}^{\alpha }
\end{equation*}%
satisfies
\begin{equation*}
dy_{t}=-\frac{1}{2}g_{\alpha }(t)g_{\beta }(t)\delta ^{\alpha \beta
}c_{\lambda }(t)dt^{\lambda }+g_{\alpha }(t)dW_{t}^{\alpha }.
\end{equation*}%
Applying the It\^{o}-Udri\c{s}te Lemma for $u(x)=e^{x}$, we obtain
\begin{equation*}
dx_{t}=\frac{\partial u}{\partial x}\,dy_{t}+\frac{1}{2}\frac{\partial ^{2}u%
}{\partial x^{2}}\,g_{\alpha }(t)g_{\beta }(t)\delta ^{\alpha \beta
}c_{\lambda }(t)dt^{\lambda }
\end{equation*}%
\begin{equation*}
=e^{y_{t}}\left( -\frac{1}{2}g_{\alpha }(t)g_{\beta }(t)\delta ^{\alpha
\beta }c_{\lambda }(t)dt^{\lambda }+g_{\alpha }(t)dW_{t}^{\alpha }+\frac{1}{2%
}g_{\alpha }(t)g_{\beta }(t)\delta ^{\alpha \beta }c_{\lambda
}(t)dt^{\lambda }\right)
\end{equation*}%
and hence
\begin{equation*}
dx_{t}=x_{t}g_{\alpha }(t)dW_{t}^{\alpha }.
\end{equation*}

3) The formula
\begin{equation*}
\frac{1}{2} W_t^2 = \frac{1}{2} W_0^2 + \frac{1}{2}\int_{\gamma_{0t}}d(s^1%
\cdots s^m) + \int_{\gamma_{0t}}W_s dW_s.
\end{equation*}
and the notations $x_t = \frac{1}{2} W_t^2, u(t) = \frac{1}{2}, v(t)= W_t$
motivate the following

\begin{definition}
A completely integrable stochastic process $x_{t}$ of the form
\begin{equation*}
x_{t}=x_{0}+\int_{\gamma _{0t}}u(s,\omega )\,d(s^{1}\cdots
s^{m})+\int_{\gamma _{0t}}v(s,\omega )\,dW_{s},
\end{equation*}%
where $u(t),\,v(t)$ satisfy the integrability conditions
\begin{equation*}
\mathbb{E}\,\int_{\gamma _{0t}}|u(s)|\,||c(s)||\,d\sigma <\infty ,\,\,%
\mathbb{E\,}\int_{\gamma _{0t}}v^{2}(s)\,d(s^{1}\cdots s^{m})<\infty ,
\end{equation*}%
is called a \textit{volumetric stochastic process}.
\end{definition}

The volumetric stochastic process can be written also as a stochastic Pfaff equation
\begin{equation*}
dx_{t}=u(t)\,d(t^{1}\cdots t^{m})+v(t)\,dW_{t}.
\end{equation*}
To motivate the adjective "volumetric" we need the following

\begin{lemma}
\textit{The curvilinear integrals $$\int_{\gamma_{0t}}u(s,\omega)\,d(s^1%
\cdots s^m)\,\, \hbox{and}\,\, \int_{\gamma_{0t}}v(s,\omega)\,dW_s$$ are path independent
if and only if $u(s,\omega) = \varphi (s^1\cdots s^m, \omega)$, respectively
$v(s,\omega)$ $ = \psi (s^1\cdots s^m, \omega)$, i.e., the functions $u$ and $v$ depend on the point $(s^{1},...,s^{m})$ only through the product of
components $s^{1}\cdots s^{m}$.}
\end{lemma}

\begin{proof}
The first curvilinear integral is path independent if and only if
\begin{equation*}
\frac{\partial}{\partial s^\beta}(u(s) c_\alpha(s)) = \frac{\partial}{%
\partial s^\alpha}(u(s) c_\beta(s)).
\end{equation*}
Since
\begin{equation*}
\frac{\partial c_\alpha}{\partial s^\beta}(s) = \frac{\partial c_\beta}{%
\partial s^\alpha}(s),
\end{equation*}
it follows
\begin{equation*}
\frac{\partial u}{\partial s^\beta}(s)c_\alpha(s) = \frac{\partial u}{%
\partial s^\alpha}(s) c_\beta(s),
\end{equation*}
i.e., $u(s) = \varphi (s^1\cdots s^m)$.

As was shown by Cairoli and Walsh \cite{cairoli&walsh}, the (second)
stochastic curvilinear integral is path independent if and only if $%
v(s)=\psi (s^{1}\cdots s^{m})$.
\end{proof}

In the following theorem we state a formula which is very useful for
computing curvilinear It\^{o} integrals. It can be thought as the
fundamental theorem of multitime stochastic calculus.

\begin{theorem}
\textit{Let $x_t, \,t = (t^1,...,t^m)\in R^m_+$ be a volumetric stochastic process, i.e.,
\begin{equation*}
dx_t = u(t)\,d(t^1\cdots t^m) + v(t)\,dW_t.
\end{equation*}
If $g(t,x)$ is a $C^2$ function, then the stochastic process $y_t = g(t,x_t)$
is completely integrable, with
\begin{equation*}
dy_t = \frac{\partial g}{\partial t^\alpha}(t,x_t) dt^\alpha + \frac{%
\partial g}{\partial x}(t,x_t) dx_t + \frac{1}{2} \frac{\partial^2 g}{%
\partial x^2}(t,x_t) (dx_t)^2,
\end{equation*}
where for computing $(dx_t)^2$ we use the following formal rules \cite%
{udriste02}
\begin{equation*}
dW_{t}~dW_{t}=d(t^1\cdots t^m)=c_\lambda(t)\,dt^\lambda,
\,dW_{t}~dt^\alpha=dt^\alpha~dW_{t}=0,\, dt^\alpha\,dt^\beta=0.
\end{equation*}%
}
\end{theorem}

Explicitly,
\begin{equation*}
dy_t = \left(\frac{\partial g}{\partial t^\alpha}(t,x_t) + \frac{\partial g}{%
\partial x}(t,x_t) u(t)c_\alpha(t) + \frac{1}{2} \frac{\partial^2 g}{%
\partial x^2}(t,x_t) v^2(t)c_\alpha(t)\right)dt^\alpha
\end{equation*}
\begin{equation*}
+ \frac{\partial g}{\partial x}(t,x_t) v(t)dW_t.
\end{equation*}
In terms of curvilinear integrals it reads as follows
\begin{equation*}
g(t,x_t) = g(0,x_0)
\end{equation*}
\begin{equation*}
+ \int_{\gamma_{0t}} \left(\frac{\partial g}{\partial s^\alpha}(s,x_s) +
\frac{\partial g}{\partial x}(s,x_s) u(s)c_\alpha(s) + \frac{1}{2} \frac{%
\partial^2 g}{\partial x^2}(s,x_s) v^2(s)c_\alpha(s)\right)ds^\alpha
\end{equation*}
\begin{equation*}
+ \int_{\gamma_{0t}}\frac{\partial g}{\partial x}(s,x_s) v(s)dW_s.
\end{equation*}

4) \textbf{Integration by parts} Let ${\gamma _{0t}}$ be an increasing curve
in $\Omega _{0T}$. If $g(t,x)=(t^{1}\cdots t^{m})\,x$, then the stochastic
process $y_{t}=g(t,x_{t})$ satisfies
\begin{equation*}
dy_{t}=x_{t}\,d(t^{1}\cdots t^{m})+(t^{1}\cdots t^{m})\,dx_{t}.
\end{equation*}%
Replacing $x_{t}=W_{t}$, we find the "\textit{integration by parts}" formula
\begin{equation*}
(t^{1}\cdots t^{m})\,W_{t}=\int_{\gamma _{0t}}W_{s}\,d(s^{1}\cdots
s^{m})+\int_{\gamma _{0t}}(s^{1}\cdots s^{m})\,dW_{s}.
\end{equation*}

Generally, if $g(t,x) = \varphi(t)\, x$, where $\varphi(t)$ is continuous
and of bounded variation in $[0, t]$, then the stochastic process $y_t =
g(t,x_t)$ satisfies
\begin{equation*}
dy_t = x_t \,d\varphi(t) + \varphi(t)\,dx_t.
\end{equation*}
Replacing $x_t = W_t$, we find the integration by parts formula
\begin{equation*}
\varphi(t)\, W_t = \int_{\gamma_{0t}} W_s \,d\varphi(s) + \int_{\gamma_{0t}}
\varphi(s) \,dW_s.
\end{equation*}

5) \textbf{Geometric Brownian Sheet} It is a stochastic process of the form
\begin{equation*}
S_t=e^{\mu _{\alpha }t^{\alpha }+\sigma B_t},\,t=(t^{1},...,t^{m})\in \Omega
_{0T}\subset
\mathbb{R}
_{+}^{m},
\end{equation*}%
where $\mu =(\mu _{\alpha })\in
\mathbb{R}
^{m}$, $\sigma >0$ and $(B_{t})_{t\in \Omega _{0T}}$ is a standard multitime
Wiener process. Since $S_t=f(t,B_t)$, for $f(t,x)=e^{\mu _{\alpha
}t^{\alpha }+\sigma x}$ and
\begin{equation*}
\frac{\partial f}{\partial t^{\alpha }}=f\mu _{\alpha },\,\frac{\partial f}{%
\partial x}=f\sigma ,\,\frac{\partial ^{2}f}{\partial x^{2}}=f\sigma ^{2},
\end{equation*}%
the It\^{o}-Udri\c{s}te formula shows that the geometric Brownian sheet is a
solution of the multitime stochastic Pfaff equation
\begin{equation*}
dS_t=\sigma S_t\,dB_t+S_t\mu _{\alpha }\,dt^{\alpha }+\frac{1}{2}\sigma
^{2}S_t\,d(t^{1}\cdots t^{m}).
\end{equation*}

\begin{definition}
\textit{(i)} Let $x(t, \omega)$ be an $m$-sheet solution of (\ref{0201}), $t\in
\Omega _{0T}$, $\omega \in \Omega$. \textit{Sheetwise uniqueness} of $%
x\left( \cdot, \omega \right) $ means that if $\overline{x}\left( \cdot,
\omega \right):\Omega_{0T}\rightarrow {\mathbb{R}}^n$ is also an $m$-sheet
solution of (\ref{0201}), on the filtered probability space endowed with the
same Wiener process and initial random variable, then
\begin{equation*}
\mathbb{P} \left[ x\left( t,\omega \right)=\overline{x}\left( t,\omega
\right), \forall t\in \Omega _{0T} \right] =1.
\end{equation*}

\textit{(ii)} Let $x(\cdot,\omega)$ be a curve solution of (\ref{0201}) with $\omega
\in \Omega $. \textit{Pathwise uniqueness} of $x(\cdot,\omega)$ means that
if $\overline{x}(\cdot,\omega)$ is also a curve solution of (\ref{0201}), on
the filtered probability space endowed with the same Wiener process and
initial random variable, then
\begin{equation*}
\mathbb{P~}\left[ x(\tau,\omega)=\overline{x}(\tau,\omega),\forall \tau \in
[0, \tau_0]\right] =1.
\end{equation*}
\end{definition}

\section{Multitime It\^{o} - Udri\c ste product formula and \\adjoint stochastic
multitime Pfaff systems} \label{SistAdjuncte}

Let $\Omega _{0T}$ be the\textit{\ parallelepiped }fixed by the
diagonal opposite points $0=\left( 0,...,0\right) $ and $T=\left(
T^{1},...,T^{m}\right) $ and $t\in \Omega _{0T}$ be the
\textit{multitime}. Given a filtred probability space $\left( \Omega
,\mathcal{F},\left( \mathcal{F}_{t}\right) _{t\in \Omega
_{0T}},\mathbb{P}\right) $ satisfying the usual conditions, on which
a Wiener process $W\left( \cdot ,\omega \right) $ with values in
$\mathbb{R}^{d}$ is defined, consider a \textit{controlled multitime stochastic differential system}%
\begin{equation}
\left\{
\begin{array}{l}
\displaystyle dx_{t}^{i}=\mu _{\alpha }^{i}\left( t,x_{t},u_{t}\right)
dt^{\alpha }+\sigma _{a}^{i}\left( t,x_{t},u_{t}\right) dW_{t}^{a}, \\
\displaystyle x\left( 0\right) =a\in \mathbb{R}^{n},%
\end{array}%
\right.  \label{rel_000_02}
\end{equation}
where
\begin{equation*}
\mu \left( \cdot ,x\left( \cdot ,\omega \right) ,u\left( \cdot ,\omega
\right) \right) =(\mu _{\alpha }^{i}):\Omega _{0T}\times \mathbb{R}%
^{n}\times U\longrightarrow \mathbb{R}^{n\times m},
\end{equation*}%
\begin{equation*}
\sigma \left( \cdot ,x\left( \cdot ,\omega \right) ,u\left( \cdot ,\omega
\right) \right) =(\sigma _{a}^{i}):\Omega _{0T}\times \mathbb{R}^{n}\times
U\longrightarrow \mathbb{R}^{n\times d}
\end{equation*}%
and, for simplicity, we denote $x\left( t,\omega \right) $, respectively $%
u\left( t,\omega \right) $, by $x_{t}$ and $u_{t}$. Here, $u_{t}\in U\subset
\mathbb{R}^{k}$ is a parameter whose value we can choose in the given Borel
set $U$ at any instant multitime $t$ in order to control the process $x_{t}$%
. Thus, $u_{t}=u\left( t,\omega \right) $ is a stochastic process, called
\textit{control }(vector-valued) \textit{variable} or, simplified,\textit{\
control}. Since our decision at multitime $t$ must be based upon what has
happened up to multitime $t$, the function $\omega \longrightarrow u\left(
t,\omega \right) $ must (at least) be measurable w.r.t. $\mathcal{F}_{t}$,
i.e. the process $u_{t}$ must be $\mathcal{F}_{t}-$adapted. We also assume
that $u\left( t,\omega \right) $ is satisfying RL hypothesis. In addition we
require that $u\left( t,\omega \right) $ gives rise to a unique solution $%
x\left( t\right) =x^{\left( u\right) }\left( t\right) $ of (\ref{rel_000_02}%
) for $t\in \Omega _{0T}$, i.e., the system (\ref{rel_000_02}) is
completely integrable. Let us denote by $\mathcal{A}$ the set of all
controls with the above properties. Any $u\left( \cdot ,\omega \right) \in
\mathcal{A} $ is called also a \textit{feasible control}.

\begin{definition}
\label{adm_contr}Let $\left( \Omega ,\mathcal{F},\left( \mathcal{F}%
_{t}\right) _{t\in \mathbb{R}_{+}},\mathbb{P}\right) $ be given satisfying the
usual conditions and let $W\left( t\right) $ be a given standard
$\left( \mathcal{F}_{t}\right) _{t\in \mathbb{R}_{+}}-$Wiener process
with values in $\mathbb{R}^{d}$. A control $u\left( \cdot ,\omega \right) $
is called \textit{admissible}, and the pair $\left( x\left( \cdot ,\omega \right) ,u\left(
\cdot ,\omega \right) \right) $ is called\textit{\ admissible}, if

\begin{enumerate}
\item $u\left( \cdot ,\omega \right) \in \mathcal{A}$;

\item $x\left( \cdot ,\omega \right) $ is the unique solution of system (\ref%
{rel_000_02});

\item some additional convex constraint on the terminal state variable are
satisfied, e.g.
\begin{equation*}
x\left( T,\omega \right) \in K,
\end{equation*}
where $K$ is a given nonempty convex subset in
$\mathbb{R}^{n}$.
\end{enumerate}
\end{definition}

The set of all admissible controls is denoted by $\mathcal{A}_{ad}$. 
We assume:

\label{ipoteza01}(\textbf{H1}) $\mu _{\alpha }^{i}$, $\sigma _{a}^{i}$, and
$f_{\alpha }$ are continuous in their arguments and continuously
differentiable in $\left( x,u\right) $, for every $i=\overline{1,n}$,
$\alpha =\overline{1,m}$, $a=\overline{1,d}$;

\label{ipoteza02}(\textbf{H2}) the derivatives of $\mu _{\alpha }^{i}$ and
$\sigma _{a}^{i}$ in $\left( x,u\right) $ are bounded for every $i=\overline{1,n}$,
$\alpha =\overline{1,m}$, $a=\overline{1,d}$;

\label{ipoteza03}(\textbf{H3}) the derivatives of $f_{\alpha }$ in $\left(
x,u\right) $ are bounded by $C\left( 1+\left\vert x\right\vert +\left\vert
u\right\vert \right) $, for every $\alpha =\overline{1,m}$ and the
derivative of $\Psi $ in $x$ is bounded by $C\left( 1+\left\vert
x\right\vert \right) $.

Then, for a given $u\left( \cdot ,\omega \right) \in \mathcal{A}_{ad}$,
there exists a unique solution $x\left( \cdot ,\omega \right) $ which solves the system
(\ref{rel_000_02}).

\subsection{Multitime It\^{o} - Udri\c ste product formula}

In order to prove the multitime stochastic maximum principle using the ideas
rising from the papers \cite{udriste02}, \cite{udr04}, \cite{udr0301}, \cite%
{udr03}, we need the following auxiliary result, which is a special case of
the multitime It\^{o} - Udri\c ste formulas \cite{udriste02}. The theory
covers both the problems formulated with complete integrable stochastic
systems and problems based on nonintegrable systems.

\begin{lemma}
\emph{(It\^{o} - Udri\c{s}te product formula)} \label{lemma copy(1)}Suppose
the It\^{o} process $\left( x_{t}^{i}\right) _{t\in \Omega _{0T}},$ $i=%
\overline{1,n}$ is solution of the multitime stochastic Pfaff system
\begin{equation*}
\left\{
\begin{array}{l}
\displaystyle dx_{t}^{i}=\mu _{\alpha }^{i}\left( t,x\left( t,\omega \right)
,u\left( t,\omega \right) \right) dt^{\alpha }+\sigma _{a}^{i}\left(
t,x\left( t,\omega \right) ,u\left( t,\omega \right) \right) dW_{t}^{a}, \\
\displaystyle x\left( 0,\omega \right) =x\in \mathbb{R}^{n},%
\end{array}%
\right.
\end{equation*}%
and the It\^{o} process $\left( p_{i}\left( t\right) \right) _{t\in
\Omega _{0T}}, \,\,i=\overline{1,n}$ is solution of the multitime
stochastic Pfaff system
\begin{equation*}
\left\{
\begin{array}{l}
\displaystyle dp_{i}\left( t\right) =a_{i\alpha }\left( t,x\left( t,\omega
\right) ,u\left( t,\omega \right) \right) dt^{\alpha }+q_{ia}\left(
t,x\left( t,\omega \right) ,u\left( t,\omega \right) \right) dW_{t}^{a}, \\
\displaystyle p\left( 0,\omega \right) =p\in \mathbb{R}^{n},%
\end{array}%
\right.
\end{equation*}%
where the coefficients in both evolutions are predictable processes and $u\left( \cdot
,\omega \right)$ is an admissible control. Then
the interior product $p_{i}(t)x^{i}(t)$ is an multitime It\^{o} process and
\begin{equation*}
d\left( p_{i}\left( t\right) x^{i}\left( t\right) \right)
=p_{i}dx^{i}+x^{i}dp_{i}+q_{ib}\sigma _{a}^{i}\delta ^{ab}c_{\alpha
}(t)dt^{\alpha }.
\end{equation*}
\end{lemma}
This equality can be called the {\it stochastic differentiation formula for the interior product.}

\subsection{Adjoint multitime stochastic Pfaff systems}

For simplicity, we will omit $\omega$ as argument of processes.

\begin{definition}
\emph{(Variational multitime stochastic Pfaff system)} Let $u\left( \cdot
,\omega \right) $ be an admissible control. Let
\begin{equation*}
dx_{t}^{i}=\mu _{\alpha }^{i}\left( t,x_{t},u_{t}\right) dt^{\alpha
}+\sigma_{a}^{i}\left( t,x_{t},u_{t}\right) dW_{t}^{a}
\end{equation*}
be a multitime stochastic Pfaff evolution. The multitime stochastic system
\begin{equation}
d\xi _{t}^{i}=\left( \mu _{\alpha x^{j}}^{i}\left( t,x_{t},u_{t}\right)
dt^{\alpha }+\sigma _{ax^{j}}^{i}\left( t,x_{t},u_{t}\right)
dW_{t}^{a}\right) ~\xi ^{j}\left( t,\omega \right)  \label{evolutie}
\end{equation}
is called {\it stochastic variational multitime Pfaff system} with control $%
u\left( \cdot ,\omega \right) $.
\end{definition}

\begin{definition}
\emph{(Adjoint multitime stochastic Pfaff system)} Consider a multitime
stochastic Pfaff evolution as in (\ref{evolutie}). A linear multitime
stochastic system of the form%
\begin{equation*}
dp_{j}\left( t\right) =\left( a_{\alpha j}^{i}\left( t,x_{t},u_{t}\right)
dt^{\alpha}+q_{bj}^{i}\left( t,x_{t},u_{t}\right) dW_{t}^{b}\right)
p_{i}\left( t\right) ,~b=\overline{1,d},~i,j=\overline{1,n}
\end{equation*}
is called {\it adjoint multitime stochastic Pfaff system} if the interior
product $p_{k}\left(t\right) \xi ^{k}\left( t\right)$ is a {\it global multitime
stochastic first integral}.
\end{definition}

\begin{theorem}
The multitime stochastic Pfaff system
\begin{equation*}
dp_{j}\left( t,\omega \right) =[\left( -\mu _{\alpha x^{j}}^{i}\left(
t,x_{t},u_{t}\right) +\sigma _{ax^{k}}^{i}\left( t,x_{t},u_{t}\right)
\sigma_{bx^{j}}^{k}\left( t,x_{t},u_{t}\right) \delta
^{ab}c_{\alpha}(t)\right) dt^{\alpha }-
\end{equation*}
\begin{equation*}
-\sigma _{ax^{j}}^{i}\left( t,x_{t},u_{t}\right) dW_{t}^{a}]~p_{i}\left(
t,\omega \right)
\end{equation*}
is the adjoint multitime stochastic Pfaff system with respect to the
variational multitime stochastic Pfaff system.
\end{theorem}

\begin{proof}
Let
\begin{equation*}
d\xi _{t}^{i}=\left( \mu _{\alpha x^{j}}^{i}\left( t,x_{t},u_{t}\right)
dt^{\alpha }+\sigma _{ax^{j}}^{i}\left( t,x_{t},u_{t}\right)
dW_{t}^{a}\right) \xi _{t}^{j},~i,j=\overline{1,n},
\end{equation*}%
be the variational stochastic multitime Pfaff system. Denote the adjoint
multitime stochastic Pfaff system by
\begin{equation*}
dp_{j}\left( t\right) =\left( a_{\alpha j}^{i}\left( t,x_{t},u_{t}\right)
dt^{\alpha }+q_{bj}^{i}\left( t,x_{t},u_{t}\right) dW_{t}^{b}\right)
p_{i}\left( t\right) ,~i,j=\overline{1,n}.
\end{equation*}%
We determine the coefficients $a_{\alpha j}^{i}$ and $q_{bj}^{i}$ such that $%
p_{k}\left( t\right) \xi ^{k}\left( t\right) $ to be a multitime stochastic
first integral, i.e.,
\begin{equation*}
d\left( p_{k}\left( t\right) \xi ^{k}\left( t\right) \right) =0,
\end{equation*}%
where $d$ is the stochastic differential. Imposing the identity
\begin{equation*}
p_{i}\left( t\right) \xi ^{i}\left( t\right) =p_{i}\left( 0\right) \xi
^{i}\left( 0\right) ,\text{ for any $t\in \Omega _{0T},$}
\end{equation*}%
or
\begin{equation*}
0=d\left( p_{k}\left( t\right) \xi ^{k}\left( t\right) \right) =p_{i}\left(
t\right) \xi ^{j}\left( t\right) (\mu _{\alpha x^{j}}^{i}\left(
t,x_{t},u_{t}\right) +a_{\alpha j}^{i}\left( t,x_{t},u_{t}\right) +
\end{equation*}%
\begin{equation*}
+q_{ak}^{i}\left( t,x_{t},u_{t}\right) \sigma _{bx^{j}}^{k}\left(
t,x_{t},u_{t}\right) \delta ^{ab}c_{\alpha }(t))dt^{\alpha }+
\end{equation*}%
\begin{equation*}
+p_{i}\left( t\right) \xi ^{j}\left( t\right) \left( \sigma
_{ax^{j}}^{i}\left( t,x_{t},u_{t}\right) +q_{aj}^{i}\left(
t,x_{t},u_{t}\right) \right) dW_{t}^{a},
\end{equation*}%
we obtain
\begin{equation*}
a_{\alpha j}^{i}\left( t,x_{t},u_{t}\right) =-\mu _{\alpha x^{j}}^{i}\left(
t,x_{t},u_{t}\right) -q_{ak}^{i}\left( t\right) \sigma _{bx^{j}}^{k}\left(
t,x_{t},u_{t}\right) \delta ^{ab}c_{\alpha }(t),
\end{equation*}%
\begin{equation*}
q_{aj}^{i}\left( t,x_{t},u_{t}\right) =-\sigma _{ax^{j}}^{i}\left(
t,x_{t},u_{t}\right) .
\end{equation*}
\end{proof}

\section{Optimization problems with \\stochastic curvilinear integral functionals} \label{optimizare}

\textit{Multitime stochastic optimal control problems with terminal
conditions} have some features: there is a \textit{multitime diffusion system%
}, which is described by a multitime It\^{o} - Pfaff stochastic
differential system; there are some \textit{constraints} that the decisions
and/or the state are subject to; there is a \textit{criterion} that measures
the performance of the decisions. The goal is to \textit{optimize} the
criterion by selecting a \textit{nonanticipative} decision among the ones
satisfying all the constraints.

\subsection{Multitime stochastic maximum principle}

The idea to use curvilinear integrals in stochastic control theory is very recent \cite{udriste02}, \cite{Ka:05}. Our paper \cite{udriste02} refers
to multitime stochastic maximum principle, It\^ o - Udri\c ste formulas and Hamilton-Jacobi-Bellman approach, while the paper \cite{Ka:05}
obtains a stochastic curvilinear integral via Hamilton-Jacobi-Bellman approach. Also, some geometrical methods 
used in the single-time stochastic theory (see \cite{FrVi:10}), can be extended to multitime stochastic techniques.

The cost functionals of stochastic economical and/or mechanical work type are
very important for applications. This motivates to solve the multitime stochastic optimal control problem 
\begin{equation}
\underset{u\left( \cdot \right) }{\max \,\,}\mathbb{E}\left[
\int_{\gamma_{0T}}f_{\alpha }\left( t,x_{t},u_{t}\right) dt^{\alpha }+
\Psi \left(x\left( T,\omega \right) \right) \right]  \label{stoch001}
\end{equation}
subject to a multitime It\^{o} process constraint
\begin{equation}
dx_{t}^{i}=\mu _{\alpha }^{i}\left( t,x_{t},u_{t}\right) dt^{\alpha }+
\sigma_{a}^{i}\left( t,x_{t},u_{t}\right) dW_{t}^{a}, \,\,x\left( 0\right) =x_{0},~x\left( T\right) =x_{T},
 \label{formula002}
\end{equation}
where $x_{t}=(x_{t}^{i})_{i=\overline{1,n}}$ is the
{\it multitime state variable}, $u\left( t\right) \in U,~\forall t\in \Omega _{0T}$ is the
{\it multitime closed-loop control variable}, $W_{t}=\left(
W_{t}^{1},...,W_{t}^{d}\right) $ is a {\it standard multitime Wiener process}.

The problem (\ref{stoch001})-(\ref{formula002}) can often be handled by
applying a kind of "Lagrange multiplier" method, as in holonomic and
nonholonomic approach of deterministic case \cite{udr0301}, respectively,
\cite{udr03}. The {\it stochastic running cost }$\eta =f_{\alpha }\left(
t,x_{t},u_{t}\right) dt^{\alpha }$, \thinspace $\alpha =\overline{1,m}$ is a
{\it stochastic Lagrangian }$1-$\textit{form}. We introduce the {\it stochastic Lagrange multiplier}
$
p_{i}\left( t,\omega \right) \in L_{\mathcal{F}}^{2}\left( \Omega _{0T},
\mathbb{R}^{n}\right) ,\text{ }i=\overline{1,n},
$
where $L_{\mathcal{F}}^{2}\left( \Omega _{0T},\mathbb{R}^{n}\right) $ is the
space of all $\mathbb{R}^{n}-$valued adapted processes
$\left( \phi_{t}\right) _{t\in \Omega _{0T}}$ such that
\begin{equation*}
\mathbb{E}\left[ \int_{\gamma _{0t}}\left\Vert \phi \left(s(\tau),x_{s(\tau)}\right)
\right\Vert d\tau\right] <\infty \text{ for all }t\in \Omega _{0T},
\end{equation*}
where $\tau$ is the curvilinear abscissa on the increasing curve
$\gamma _{0t}$, $t\in \Omega _{0T}$.
To use geometrical methods in control theory, let us suppose
$\left(p_{i}\left( t,\omega \right) \right) _{t\in \Omega _{0T}}$ as a
multitime It\^{o} process, i.e.,
\begin{equation*}
dp_{i}\left( t\right) =a_{i\alpha }\left( t,x_{t},u_{t}\right) dt^{\alpha}+
q_{ia}\left( t,x_{t},u_{t}\right) dW_{t}^{a},
\end{equation*}
where $\left[ \left( a_{i\alpha }\left( t,x_{t},u_{t}\right) \right) _{t\in
\Omega _{0T}}\right] _{\substack{ 1\leq i\leq n  \\ 1\leq \alpha \leq m}}$,
respectively, $\left[ \left( q_{ia}\left( t,x_{t},u_{t}\right) \right)
_{t\in \Omega _{0T}}\right] _{\substack{ 1\leq i\leq n  \\ 1\leq a\leq d}}$
are matrices of previsible processes. The {\it adjoint process} $\left(
p_{i}\left( t,\omega \right) \right) _{i=\overline{1,n}}$\ is required to be
$\left( \mathcal{F}_{t}\right) _{t\in \Omega _{0T}}-$adapted, for any $t\in
\Omega _{0T}$.

To solve the foregoing problem, we use the {\it stochastic Lagrangian $1-$form}
$$
\mathcal{L}\left( t,x_{t},u_{t},p_{t}\right) =f_{\alpha }\left(
t,x_{t},u_{t}\right) dt^{\alpha }
$$
$$
+p_{i}\left( t\right) \left[ \mu _{\alpha
}^{i}\left( t,x_{t},u_{t}\right) dt^{\alpha }+\sigma _{a}^{i}\left(
t,x_{t},u_{t}\right) dW_{t}^{a}-dx_{t}^{i}\right] .
$$

The foregoing multitime stochastic optimization problem, constrained by a contact distribution, with stochastic perturbations, (\ref{stoch001})-(\ref{formula002}) can be
change into another free multitime stochastic optimization problem%
\begin{equation}
\underset{u\left( \cdot ,\omega \right) \in \mathcal{A}_{ad}}{\max }\mathbb{E%
}\left[ \int_{\gamma _{0T}}\mathcal{L}\left( t,x_{t},u_{t},p_{t}\right)
+\Psi \left( x\left( T,\omega \right) \right) \right] ,  \label{problem}
\end{equation}%
\begin{equation*}
\text{subject to}
\end{equation*}%
\begin{equation*}
p\left( t,\omega \right) \in \mathcal{P},\mathcal{~}\forall t\in \Omega
_{0T},~x\left( 0,\omega \right) =x_{0}\in
\mathbb{R}
^{n},
\end{equation*}%
where the set $\mathcal{P}$ will be defined as the set of adjoint multitime
stochastic processes in an appropriate context. The problem (\ref{problem})
can be rewritten as%
\begin{equation}
\underset{u\left( \cdot ,\omega \right) \in \mathcal{A}_{ad}}{\max }\mathbb{E%
}\left\{ \int_{\gamma _{0T}}\left[ f_{\alpha }\left( t,x_{t},u_{t}\right)
+p_{i}\left( t\right) \mu _{\alpha }^{i}\left( t,x_{t},u_{t}\right) \right]
dt^{\alpha }\right.
\end{equation}%
\begin{equation*}
+\left. \int_{\gamma _{0T}}p_{i}\left( t\right) \sigma _{a}^{i}\left(
t,x_{t},u_{t}\right) dW_{t}^{a}-\int_{\gamma _{0T}}p_{i}\left( t\right)
dx_{t}^{i}+\Psi \left( x\left( T,\omega \right) \right) \right\} ,
\end{equation*}%
\begin{equation*}
\text{subject to}
\end{equation*}%
\begin{equation*}
p\left( t,\omega \right) \in \mathcal{P},\mathcal{~}\forall t\in \Omega
_{0T},~x\left( 0,\omega \right) =x_{0}\in
\mathbb{R}
^{n},~i=\overline{1,n}.
\end{equation*}

Due to properties of stochastic curvilinear integrals (Remark \ref%
{remarca_01}), the terms containing $dW_t^a$ vanish. Evaluating $\int_{\gamma _{0T}}p_{i}\left( t,\omega \right) dx_{t}^{i}$,
via multitime stochastic integration by parts, it appears the {\it control Hamiltonian multitime stochastic $1-$form}
$$
\mathcal{H}\left( t,x_{t},u_{t},p_{t}\right) =\left(f_{\alpha }\left(t,x_{t},u_{t}\right) +p_{i}\left( t\right) \mu _{\alpha}^{i}\left( t,x_{t},u_{t}\right)\right. 
$$
$$
\left. -p_{i}\left( t\right) \sigma _{x^{j}b}^{i}\left( t,x_{t},u_{t}\right)\sigma _{a}^{j}\left( t,x_{t},u_{t}\right) \delta^{ab}c_{\alpha}(t)\right)dt^{\alpha}.
$$
It verifies a {\it modified multitime stochastic Legendrian duality}, i.e.,

\begin{eqnarray*}
\mathcal{H} &\mathcal{=}&\mathcal{L+}p_{i}\left( t\right)
dx_{t}^{i}-p_{i}\left( t\right) \sigma _{ax^{j}}^{i}\left(
t,x_{t},u_{t}\right) \sigma _{b}^{i}\left( t,x_{t},u_{t}\right) \delta
^{ab}c_{\alpha }(t)dt^{\alpha } \\
&&-p_{i}\left( t\right) \sigma _{a}^{i}\left( t,x_{t},u_{t}\right)
dW_{t}^{a}.
\end{eqnarray*}

The key point to the derivation of the necessary
conditions of optimality is that the multitime Legendre stochastic transformation of the multitime Lagrangian stochastic $1$-form to
be minimized into a multitime Hamiltonian stochastic $1$-form converts a functional minimization problem into a static
optimization problem on the $1$-form $\mathcal{H}\left( t,x_{t},u_{t},p_{t}\right)$.

\begin{theorem}[Stochastic maximum principle]
\label{simp_stoch_max_princ}We assume the conditions $H1-H3$. Suppose that
the problem of maximizing the functional (\ref{stoch001}) constrained by (\ref{formula002}) over $\mathcal{A}_{ad}$ has an interior optimal solution $%
u^{\ast }\left( t\right) $, which determines the multitime stochastic
optimal evolution $x\left( t\right) $. Let $\mathcal{H}$ be the Hamiltonian
multitime stochastic $1-$form. Then there exists an adapted process $\left(
p\left( t,\omega \right) \right) _{t\in \Omega _{0T}}$ (adjoint process)
satisfying:

(i) the initial multitime stochastic differential system,
\begin{eqnarray*}
dx^{i}\left( t\right) &=&\frac{\partial \mathcal{H}}{\partial p_{i}}\left(
t,x_{t},u_{t}^{\ast },p_{t}\right) +\sigma _{ax^{j}}^{i}\left(
t,x_{t},u_{t}^{\ast }\right) \sigma _{b}^{i}\left( t,x_{t},u_{t}^{\ast
}\right) \delta ^{ab}c_{\alpha }(t)dt^{\alpha } \\
&&+\sigma _{a}^{i}\left( t,x_{t},u_{t}^{\ast }\right) dW_{t}^{a};
\end{eqnarray*}

(ii) the linear multitime stochastic adjoint system,
\begin{eqnarray*}
dp_{i}\left( t\right) &=&-\mathcal{H}_{x^{i}}\left( t,x_{t},u_{t}^{\ast
},p_{t}\right) -p_{j}\left( t\right) \sigma _{ax^{i}}^{j}\left(
t,x_{t},u_{t}^{\ast }\right) dW_{t}^{a}, \\
p_{i}\left( T\right) &=&\Psi _{x^{i}}\left( x_{T}\right) ,~i=\overline{1,n};
\end{eqnarray*}

(iii) the critical point condition,

\begin{equation*}
\mathcal{H}_{u^{c}}\left( t,x_{t},u_{t}^{\ast },p_{t}\right) =0,~c=\overline{%
1,k}.
\end{equation*}
\end{theorem}

\begin{proof}
In the whole this proof, we will omit $\omega $ as argument of
processes. Suppose that there exists a continuous control $u^{\ast
}\left( t,\omega \right) $ over the admissible controls in
$\mathcal{A}_{ad}$, which is an
optimum point in the previous problem. Consider a variation%
\begin{equation*}
u\left( t,\varepsilon \right) =u^{\ast }\left( t\right) +\varepsilon h\left(
t\right) ,
\end{equation*}%
where, by hypothesis, $h\left( t\right) $ is an arbitrary multitime
stochastic process. Since $u^{\ast }\left( t,\omega \right) \in \mathcal{A}%
_{ad}$ and a continuous function over a compact set $\Omega _{0T}$ is
bounded, there exists a number $\varepsilon _{h}>0$ such that%
\begin{equation*}
u\left( t,\varepsilon \right) =u^{\ast }\left( t\right) +\varepsilon h\left(
t\right) \in \mathcal{A}_{ad},~\forall \left\vert \varepsilon \right\vert
<\varepsilon _{h}.
\end{equation*}%
This $\varepsilon $ is used in our variational arguments.

Now, let us define the contact distribution with stochastic multitime
perturbations, corresponding to the control variable $u\left( t,\varepsilon
\right) $, i.e.,%
\begin{equation*}
dx^{i}\left( t,\varepsilon \right) =\mu _{\alpha }^{i}\left( t,x\left(
t,\varepsilon \right) ,u\left( t,\varepsilon \right) \right) dt^{\alpha
}+\sigma _{a}^{i}\left( t,x\left( t,\varepsilon \right) ,u\left(
t,\varepsilon \right) \right) dW_{t}^{a},
\end{equation*}%
for $i=\overline{1,n}$, or,%
\begin{equation*}
x^{i}\left( t,\varepsilon \right) =x^{i}\left( 0,\varepsilon \right)
+\int_{\gamma _{0t}}\mu _{\alpha }^{i}\left( s,x\left( s,\varepsilon \right)
,u\left( s,\varepsilon \right) \right) ds^{\alpha }
\end{equation*}%
\begin{equation*}
+\int_{\gamma _{0t}}\sigma _{a}^{i}\left( s,x\left( s,\varepsilon \right)
,u\left( s,\varepsilon \right) \right) dW_{s}^{a},~\forall t\in \Omega
_{0T},~\forall i=\overline{1,n}
\end{equation*}%
and $x\left( 0,\varepsilon \right) =x_{0}\in
\mathbb{R}
^{n}$. For $\left\vert \varepsilon \right\vert <\varepsilon _{h}$, we define
the function%
\begin{equation*}
J\left( \varepsilon \right) =\mathbb{E}\left[ \int_{\gamma _{0T}}f_{\alpha
}\left( t,x\left( t,\varepsilon \right) ,u\left( t,\varepsilon \right)
\right) dt^{\alpha }+\Psi \left( x\left( T,\varepsilon \right) \right) %
\right] .
\end{equation*}

For any adapted process $\left( p_{t}\right) _{t\in \Omega _{0T}}$ we have%
\begin{equation*}
\int_{\gamma _{0T}}p_{i}\left( t\right) ~\left[ \mu _{\alpha }^{i}\left(
t,x\left( t,\varepsilon \right) ,u\left( t,\varepsilon \right) \right)
dt^{\alpha }-dx^{i}\left( t,\varepsilon \right) \right]
\end{equation*}%
\begin{equation*}
+\int_{\gamma _{0T}}p_{i}\left( t\right) ~\sigma _{a}^{i}\left( t,x\left(
t,\varepsilon \right) ,u\left( t,\varepsilon \right) \right) dW_{t}^{a}=0,~i=%
\overline{1,n}.
\end{equation*}

To solve the forgoing constrained multitime stochastic optimization problem,
we transform it into a free multitime stochastic optimization problem
(\cite{udr0301}). For that, we use the Lagrange multitime stochastic $1$-form
which includes the variations%
\begin{equation*}
\mathcal{L}\left( t,x\left( t,\varepsilon \right) ,u\left( t,\varepsilon
\right) ,p_{t}\right) =f_{\alpha }\left( t,x\left( t,\varepsilon \right)
,u\left( t,\varepsilon \right) \right) dt^{\alpha }
\end{equation*}%
\begin{equation*}
+p_{i}\left( t\right) \left[ \mu _{\alpha }^{i}\left( t,x\left(
t,\varepsilon \right) ,u\left( t,\varepsilon \right) \right) dt^{\alpha
}+\sigma _{a}^{i}\left( t,x\left( t,\varepsilon \right) ,u\left(
t,\varepsilon \right) \right) dW_{t}^{a}-dx^{i}\left( t,\varepsilon \right) %
\right] ,
\end{equation*}%
where $i=\overline{1,n}$. We have to optimize now the function%
\begin{equation*}
\widetilde{J}\left( \varepsilon \right) =\mathbb{E}\left[ \int_{\gamma _{0T}}%
\mathcal{L}\left( t,x\left( t,\varepsilon \right) ,u\left( t,\varepsilon
\right) ,p_{t}\right) +\Psi \left( x\left( T,\varepsilon \right) \right) %
\right] ,
\end{equation*}%
with doubt any constraints. If the control $u^{\ast }\left( t\right) $ is
optimal, then%
\begin{equation*}
\widetilde{J}\left( \varepsilon \right) \leq \widetilde{J}\left( 0\right)
,~\forall \left\vert \varepsilon \right\vert <\varepsilon _{h}.
\end{equation*}

Explicitly,%
\begin{equation*}
\widetilde{J}\left( \varepsilon \right) =\mathbb{E}\int_{\gamma _{0T}}\left[
f_{\alpha }\left( t,x\left( t,\varepsilon \right) ,u\left( t,\varepsilon
\right) \right) +p_{i}\left( t\right) \mu _{\alpha }^{i}\left( t,x\left(
t,\varepsilon \right) ,u\left( t,\varepsilon \right) \right) \right]
dt^{\alpha }
\end{equation*}%
\begin{equation*}
-\mathbb{E}\left[ \int_{\gamma _{0T}}p_{i}\left( t,\omega \right)
dx^{i}\left( t,\omega ,\varepsilon \right) +\Psi \left( x\left(
T,\varepsilon \right) \right) \right] ,~i=\overline{1,n}.
\end{equation*}

To evaluate the integral%
\begin{equation*}
\int_{\gamma _{0T}}p_{i}\left( t\right) ~dx^{i}\left( t,\varepsilon \right) ,
\end{equation*}%
we integrate by parts, via It\^{o}-Udri\c ste product formula.
Taking into account that $\left( p_{t}\right) _{t\in \Omega _{0T}}$ is a
multitime It\^{o} process, we obtain%
\begin{equation*}
\widetilde{J}\left( \varepsilon \right) =\mathbb{E}\int_{\gamma _{0T}}\left[
f_{\alpha }\left( t,x\left( t,\varepsilon \right) ,u\left( t,\varepsilon
\right) \right) +p_{i}\left( t\right) \mu _{\alpha }^{i}\left( t,x\left(
t,\varepsilon \right) ,u\left( t,\varepsilon \right) \right) \right]
dt^{\alpha }
\end{equation*}%
\begin{equation*}
-\mathbb{E}\left[ \left. p_{i}\left( t\right) ~x^{i}\left( t,\varepsilon
\right) \right\vert _{0}^{T}-\int_{\widetilde{\gamma }_{0T}}x^{i}\left(
t,\varepsilon \right) ~dp_{i}\left( t\right) \right] +\mathbb{E}\Psi \left( x\left( T,\varepsilon \right) \right) 
\end{equation*}%
\begin{equation*}
-\mathbb{E}\int_{\gamma _{0T}}p_{j}\left( t\right) \sigma
_{ax^{i}}^{j}\left( t,x\left( t,\varepsilon \right) ,u\left( t,\varepsilon
\right) \right) \sigma _{b}^{i}\left( t,x\left( t,\varepsilon \right)
,u\left( t,\varepsilon \right) \right) \delta ^{ab}c_{\alpha }dt^{\alpha },
\end{equation*}%
with $\sigma _{ax^{i}}^{j}\left( t,x\left( t,\varepsilon \right) ,u\left(
t,\varepsilon \right) \right) \mid _{\varepsilon =0}\equiv \sigma
_{ax^{i}}^{j}\left( t,x_{t},u_{t}\right) $, for all $i,j=\overline{1,n}$ and
$a=\overline{1,d}$ and where $\widetilde{\gamma }_{0T}$ is a inverse image
of the projection $\pi \circ \widetilde{\gamma }_{0T}=\gamma _{0T}$. Then,%
\begin{equation*}
\widetilde{J}\left( \varepsilon \right) =\mathbb{E}\left[ \int_{\gamma
_{0T}}f_{\alpha }\left( t,x\left( t,\varepsilon \right) ,u\left(
t,\varepsilon \right) \right) dt^{\alpha }\right]
\end{equation*}%
\begin{equation*}
+\mathbb{E}\int_{\gamma _{0T}}[p_{i}\left( t\right) \mu _{\alpha }^{i}\left(
t,x\left( t,\varepsilon \right) ,u\left( t,\varepsilon \right) \right)
\end{equation*}%
\begin{equation*}
-p_{j}\left( t\right) \sigma _{ax^{i}}^{j}\left( t,x\left( t,\varepsilon
\right) ,u\left( t,\varepsilon \right) \right) \sigma _{b}^{i}\left(
t,x\left( t,\varepsilon \right) ,u\left( t,\varepsilon \right) \right)
\delta ^{ab}c_{\alpha }]~dt^{\gamma }
\end{equation*}%
\begin{equation*}
+\mathbb{E~}\left[ \int_{\widetilde{\gamma }_{0T}}x^{i}\left( t,\varepsilon
\right) ~dp_{i}\left( t\right) -\left. p_{i}\left( t\right) ~x^{i}\left(
t,\varepsilon \right) \right\vert _{0}^{T}\right] +\mathbb{E}\Psi \left(
x\left( T,\varepsilon \right) \right).
\end{equation*}

Differentiating with respect to $\varepsilon $ (and this is possible,
because the derivative with respect to $\varepsilon $, for $\varepsilon =0$,
exists in mean square sense; see, for example \cite{var}), it follows

\begin{equation*}
\widetilde{J}^{\prime }\left( \varepsilon \right) =\mathbb{E}\int_{\gamma
_{0T}}\{[f_{\alpha x^{k}}\left( t,x\left( t,\varepsilon \right) ,u\left(
t,\varepsilon \right) \right) +p_{i}\left( t\right) \mu _{\alpha
x^{k}}^{i}\left( t,x\left( t,\varepsilon \right) ,u\left( t,\varepsilon
\right) \right)
\end{equation*}%
\begin{equation*}
-p_{j}\left( t\right) \sigma _{ax^{i}x^{k}}^{j}\left( t,x\left(
t,\varepsilon \right) ,u\left( t,\varepsilon \right) \right) \sigma
_{b}^{i}\left( t,x\left( t,\varepsilon \right) ,u\left( t,\varepsilon
\right) \right) \delta ^{ab}c_{\alpha }
\end{equation*}%
\begin{equation*}
-p_{j}\left( t\right) \sigma _{ax^{i}}^{j}\left( t,x\left( t,\varepsilon
\right) ,u\left( t,\varepsilon \right) \right) \sigma _{bx^{k}}^{i}\left(
t,x\left( t,\varepsilon \right) ,u\left( t,\varepsilon \right) \right)
\delta ^{ab}c_{\alpha }]~dt^{\alpha }+dp_{k}\left( t\right)\}\,x_{\varepsilon }^{k}
\end{equation*}%
\begin{equation*}
+\mathbb{E}\int_{\gamma _{0T}}[f_{\alpha u^{c}}\left( t,x\left(
t,\varepsilon \right) ,u\left( t,\varepsilon \right) \right) +p_{i}\left(
t\right) \mu _{\alpha u^{c}}^{i}\left( t,x\left( t,\varepsilon \right)
,u\left( t,\varepsilon \right) \right)
\end{equation*}%
\begin{equation*}
-p_{j}\left( t\right) \sigma _{ax^{i}u^{c}}^{j}\left( t,x\left(
t,\varepsilon \right) ,u\left( t,\varepsilon \right) \right) \sigma
_{b}^{i}\left( t,x\left( t,\varepsilon \right) ,u\left( t,\varepsilon
\right) \right) \delta ^{ab}c_{\alpha }
\end{equation*}%
\begin{equation*}
-p_{j}\left( t\right) \sigma _{ax^{i}}^{j}\left( t,x\left( t,\varepsilon
\right) ,u\left( t,\varepsilon \right) \right) \sigma _{bu^{c}}^{i}\left(
t,x\left( t,\varepsilon \right) ,u\left( t,\varepsilon \right) \right)
\delta ^{ab}c_{\alpha }]~h^{c}\left( t\right) ~dt^{\alpha }
\end{equation*}%
\begin{equation*}
+\mathbb{E}\Psi _{x^{k}}\left( x\left( T,\varepsilon \right) \right)
x_{\varepsilon }^{k}\left( T,\varepsilon \right) ,~c=\overline{1,k}.
\end{equation*}

Evaluating at $\varepsilon =0$, we find%
\begin{equation*}
\widetilde{J}^{\prime }\left( 0\right) =\mathbb{E}\int_{\gamma
_{0T}}\{[f_{\alpha x^{k}}\left( t,x_{t},u_{t}\right) +p_{i}\left( t\right)
\mu _{\alpha x^{k}}^{i}\left( t,x_{t},u_{t}\right)
\end{equation*}%
\begin{equation*}
-p_{j}\left( t\right) \sigma _{ax^{i}x^{k}}^{j}\left( t,x_{t},u_{t}\right)
\sigma _{b}^{i}\left( t,x_{t},u_{t}\right) \delta ^{ab}c_{\alpha }
\end{equation*}%
\begin{equation*}
-p_{j}\left( t\right) \sigma _{ax^{i}}^{j}\left( t,x_{t},u_{t}\right) \sigma
_{bx^{k}}^{i}\left( t,x_{t},u_{t}\right) \delta ^{ab}c_{\alpha }]~dt^{\alpha
}+dp_{k}\left( t\right) \}~x_{\varepsilon }^{k}\left( t,0\right)
\end{equation*}%
\begin{equation*}
+\mathbb{E}\int_{\gamma _{0T}}[f_{\alpha u^{c}}\left( t,x_{t},u_{t}^{\ast
}\right) +p_{i}\left( t\right) \mu _{\alpha u^{c}}^{i}\left(
t,x_{t},u_{t}^{\ast }\right)
\end{equation*}%
\begin{equation*}
-p_{j}\left( t\right) \sigma _{ax^{i}u^{c}}^{j}\left( t,x_{t},u_{t}^{\ast
}\right) \sigma _{b}^{i}\left( t,x_{t},u_{t}^{\ast }\right) \delta
^{ab}c_{\alpha }
\end{equation*}%
\begin{equation*}
-p_{j}\left( t\right) \sigma _{ax^{i}}^{j}\left( t,x_{t},u_{t}^{\ast
}\right) \sigma _{bu^{c}}^{i}\left( t,x_{t},u_{t}^{\ast }\right) \delta
^{ab}c_{\alpha }]~h^{c}\left( t\right) ~dt^{\alpha }
\end{equation*}%
\begin{equation*}
+\mathbb{E}\Psi _{x^{k}}\left( x\left( T\right) \right) x_{\varepsilon
}^{k}\left( T\right) ,~c=\overline{1,k}.
\end{equation*}%
where $x\left( t\right) $ is the multitime state variable corresponding to
the optimal control $u^{\ast }\left( t\right) $.

We need $J^{\prime }\left( 0\right) =0$ for all $h\left( t\right) =\left(
h^{c}\left( t\right) \right) _{c=\overline{1,k}}$. On the other hand, the
functions $x_{\varepsilon }^{i}\left( t,0\right) $ are involved in the
Cauchy problem%
\begin{equation*}
dx_{\varepsilon }^{i}\left( t,0\right) =\left( \mu _{\alpha x^{j}}^{i}\left(t,x\left( t,0\right) ,u\left( t,0\right) \right) dt^{\alpha}
+\sigma_{ax^{j}}^{i}\left( t,x_{t},u_{t}\right) dW_{t}^{a}\right) x_{\varepsilon}^{j}\left( t,0\right)
\end{equation*}%
\begin{equation*}
+\left( \mu _{\alpha u^{c}}^{i}\left( t,x\left( t,0\right) ,u\left(
t,0\right) \right) dt^{\alpha }+\sigma _{bu^{c}}^{i}\left(
t,x_{t},u_{t}\right) dW_{t}^{b}\right) h^{c}\left( t\right) ,
\end{equation*}%
\begin{equation*}
t\in \Omega _{0T},~x_{\varepsilon }\left( 0,0\right) =0\in
\mathbb{R}
^{n}
\end{equation*}%
and hence they depend on $h\left( t\right) $. The functions $x_{\varepsilon
}^{i}\left( t,0\right) $ are eliminated by selecting $\mathcal{P}$ as the
adjoint contact distribution%
\begin{eqnarray}
dp_{k}\left( t\right)  &=&-[f_{\alpha x^{k}}\left( t,x_{t},u_{t}\right)
+p_{i}\left( t\right) \mu _{\alpha x^{k}}^{i}\left( t,x_{t},u_{t}\right)
\label{relatia4} \\
&&-p_{j}\left( t\right) \sigma _{ax^{i}x^{k}}^{j}\left( t,x_{t},u_{t}\right)
\sigma _{b}^{i}\left( t,x_{t},u_{t}\right) \delta ^{ab}c_{\alpha }\notag
\\
&&-p_{j}\left( t\right) \sigma _{ax^{i}}^{j}\left( t,x_{t},u_{t}\right)
\sigma _{bx^{k}}^{i}\left( t,x_{t},u_{t}\right) \delta ^{ab}c_{\alpha
}]dt^{\alpha }\notag\\
&&-\sigma _{ax^{k}}^{j}\left( t,x_{t},u_{t}\right) p_{j}\left( t\right)dW_{t}^{a},\notag
\end{eqnarray}%
for any multitime $t\in \Omega _{0T}$, with stochastic perturbations terminal value
problem (see, e.g. \cite{udr001})%
\begin{equation*}
p_{k}\left( T\right) =\Psi _{x^{k}}\left( x_{T}\right) ,~k=\overline{1,n}.
\end{equation*}%
The relation (\ref{relatia4}) shows that%
\begin{eqnarray*}
a_{i\alpha }\left( t,x_{t},u_{t}\right)  &=&-f_{\alpha x^{k}}\left(
t,x_{t},u_{t}\right) -p_{i}\left( t\right) \mu _{\alpha x^{k}}^{i}\left(
t,x_{t},u_{t}\right) \\
&&+[p_{j}\left( t\right) \sigma _{ax^{i}x^{k}}^{j}\left(
t,x_{t},u_{t}\right) \sigma _{b}^{i}\left( t,x_{t},u_{t}\right) c_{\alpha }
\\
&&+p_{j}\left( t\right) \sigma _{ax^{i}}^{j}\left( t,x_{t},u_{t}\right)
\sigma _{bx^{k}}^{i}\left( t,x_{t},u_{t}\right) c_{\alpha }]\delta ^{ab}.
\end{eqnarray*}%
It follows%
\begin{equation*}
dp_{i}\left( t,\omega \right) =-\mathcal{H}_{x^{i}}\left(
t,x_{t},u_{t}^{\ast },p_{t}\right) -p_{j}\left( t\right) \sigma
_{ax^{i}}^{j}\left( t,x_{t},u_{t}\right) dW_{t}^{a}\text{,}
\end{equation*}%
\begin{equation*}
\forall t\in \Omega _{0T}\text{, }i=\overline{1,n},
\end{equation*}%
and%
\begin{equation*}
\mathcal{H}_{u^{c}}\left( t,x_{t},u_{t}^{\ast },p_{t}\right) =0\text{, }%
\forall t\in \Omega _{0T},\text{ for all }c=\overline{1,k}.
\end{equation*}%
Moreover,%
\begin{eqnarray*}
dx^{i}\left( t\right)  &=&\frac{\partial \mathcal{H}}{\partial p_{i}}\left(
t,x_{t},u_{t}^{\ast },p_{t}\right) +\sigma _{ax^{j}}^{i}\left(
t,x_{t},u_{t}^{\ast }\right) \sigma _{b}^{j}\left( t,x_{t},u_{t}^{\ast
}\right) \delta ^{ab}c_{\alpha }dt^{\alpha } \\
&&+\sigma _{a}^{i}\left( t,x_{t},u_{t}^{\ast }\right) dW_{t}^{a}\text{,~}%
\forall t\in \Omega _{0T},x\left( 0\right) =x_{0},~\forall i=\overline{1,n}.
\end{eqnarray*}
\end{proof}

{\bf Example} ({\bf Deterministic problem, continuous two-time version},
see \cite{udris2011}) A mine owner must decide at what rate to
extract a complex ore from his mine. He owns rights to the ore from
two-time date $0= (0,0)$ to two-time date $T= (T,T)$. A {\it
two-time date} can be the pair (date, useful component frequency).
At two-time date $0$ there is $x_0 = (x^i_0)$ ore in the ground, and
the instantaneous stock of ore $x(t) = (x^i(t))$ declines at the
rate $u(t) = (u^i_\alpha(t))$ the mine owner extracts it. The mine
owner extracts ore at cost $q_i\frac{{u^i_\alpha(t)}^2}{x^i(t)}$ and
sells ore at a constant price $p= (p_i)$. He does not value the ore
remaining in the ground at time $T$ (there is no "scrap value"). He
chooses the rate $u(t)$ of extraction in two-time to maximize
profits over the period of ownership with no two-time discounting.

{\bf Stochastic model} Since there must be some risk in the investment, the deterministic evolution
$\frac{\partial x}{\partial t^\gamma}(t) = - u_\gamma(t)$ can be changed into a stochastic model.
Hence, the manager want to maximizes the profit (curvilinear integral)
$$P(u(\cdot)) = \int_{\gamma_{0T}}\left(p_iu^i_\alpha(t)- q_i\frac{{u^i_\alpha(t)}^2}{x^i(t)}\right)dt^\alpha$$
subject to the stochastic law of evolution
$$\frac{\partial x}{\partial t^\gamma}(t) = - u_\gamma(t) + W_{t^\gamma},$$
where $W_t = (W^i_{t^\gamma})$ is a matrix Wiener process. Form the Hamiltonian $1$-form
$${\cal H} = \left(p_i u^i_\alpha(t)- q_i\frac{{u^i_\alpha(t)}^2}{x^i(t)} -\lambda_i(t)u^i_\alpha(t)\right)dt^\alpha,$$
differentiate and write the equations
$$\frac{\partial {\cal H}}{\partial u_\beta}= \left(p_i  - 2 q_i\frac{{u^i_\alpha}}{x^i}- \lambda_i\right)\delta^\beta_\alpha\, dt^\alpha = 0,\,\,\hbox{no\, sum\, after\,\,the\,\,index\,\, i};$$
$$d \lambda_i(t) = - \frac{\partial {\cal H}}{\partial x^i} = - q_i \left(\frac{u^i_\alpha(t)}{x^i(t)}\right)^2\,dt^\alpha,\,\,\hbox{no\, sum\, after\,\,the\,\,index\,\, i}.$$
As the mine owner does not value the ore remaining at time $T$, we have $\lambda_i(T) = 0$.
Using the above equations, it is easy to solve for the differential equations governing the control vector $u(t)$ and the dual vector $\lambda_i(t)$:
$$2 q_i\frac{{u^i_\alpha(t)}}{x^i(t)} = p_i  -  \lambda_i(t),\,\,\frac{\partial \lambda_i}{\partial t^\alpha}(t) =  - q_i \left(\frac{u^i_\alpha(t)}{x^i(t)}\right)^2\,\,\hbox{no\, sum\, after\,\, i}$$
and using the initial and turn - $T$ conditions, the equations can be solved numerically.

{\bf Acknowledgement}
In our work we have benefitted from useful suggestions from Prof. Dr. Ionel \c Tevy (University \textit{Politehnica} Bucharest) and
Prof. Dr. Constantin V\^ arsan (Institute of Mathematics \textit{Simion Stoilow} of the Romanian Academy).
We thank them for contributions to the improvement of our paper.



\end{document}